\documentclass[12pt]{article}

\usepackage{amsfonts}
\usepackage{amsmath}
\newtheorem{definition}{Definition}
\newtheorem{theorem}{Theorem}

\newtheorem{lemma}{Lemma}
\newtheorem{remark}{Remark}

\newtheorem{example}{Example}
\newenvironment{proof}{\noindent{\it Proof. }\rm}
{\unskip\nobreak\hfil\penalty50\hskip1em\hbox{}
\nobreak\hfill\qed\par\smallskip}
\def\qed{\vrule height1ex width1ex depth0pt}

\begin{document}
\title{Shilov boundary for "holomorphic functions"  on a quantum matrix ball}
%\author{D.P. L.T.}
\author{Daniil Proskurin \vspace{0.1cm}\\
{\footnotesize \sl Kyiv Taras Shevchenko University, Cybernetics Department,
Volodymyrska, 64, }\\
{\footnotesize \sl 01033, Kyiv, Ukraine}\\
 {\footnotesize \texttt{prosk@univ.kiev.ua}}\vspace{0.3cm}\\
Lyudmila  Turowska \vspace{0.1cm}\\
{\footnotesize \sl Department of Mathematics, Chalmers University of
Technology and University of Gothenburg, }\\
{\footnotesize \sl 412 96 G\"oteborg, Sweden} \\
{\footnotesize \texttt{turowska@chalmers.se}}}

\date{}
\maketitle{}
\abstract{We describe the Shilov boundary ideal for a q-analog of algebra of holomorphic functions on the unit ball in the space of $2\times 2$ matrices.}

\bigskip
\section{Introduction}
The Shilov boundary of a compact Hausdorff space $X$ relative  a uniform algebra $\mathcal A$ in $C(X)$  is the smallest closed subset $K\subset X$ such that every function in $\mathcal A$ achieves its maximum modulus on $K$, a notion that is closely related to the maximum modules principle in complex analysis.

 One of the most important developments in Analysis in recent years has been "quantisation", starting with the advent of the theory of operator spaces in the 1980's.
 A quantisation of the Shilov boundary is
a Shilov boundary ideal of a $C^*$-algebra, that was introduced by W. Arveson in his foundational papers
\cite{arveson, arveson2} and studied intensively by many authors.

In the middle of 1990's within the framework of the quantum group theory L.Vaksman and his coauthors started a "quantisation" of bounded symmetric domains (see \cite{vaksman-book} and references therein). One of the simplest of such domains is  the matrix ball
$\mathbb U=\{z\in Mat_{m,n}: zz^*\leq I\}$, where $Mat_{m,n}$ is the algebra of complex $m\times n$ matrices. Its $q$-analog was studied in \cite{sv, vaksman-matrix ball} where the authors defined a non-commutative counterpart of the polynomial algebra  in the space $Mat_{m,n}$, the  $*$-algebra $Pol(Mat_{m,n})_q$.
A $q$-analog of polynomial algebra on the Shilov boundary $S(\mathbb U)$ of the matrix ball $\mathbb U$ and the corresponding Cauchy-Szeg\"o integral representation that recovers holomorphic functions from its values on the Shilov boundary was studied in \cite{vaksman-matrix ball}. The authors used a purely algebraic approach "quantizing" a well known procedure for producing the Shilov boundary in the classical case.  In particular, they constructed a $*$-homomorphism $\psi:Pol(Mat_{m,n})_q\to Pol(S(\mathbb U))_q$ that corresponds to the restriction of the polynomials onto the Shilov boundary in the classical case. That the kernel of $\psi$ gives rise to Arveson's Shilov boundary ideal for the $q$-analog of holomorphic functions on the unit ball
of $Mat_{n,1}=\mathbb C^n$ was shown in \cite{vaksman-boundary}. In this paper we prove the statement for the case $m=n=2$.

Our approach relies on a classification of irreducible representations of $Pol(Mat_{2,2})_q $ obtained in \cite{T} and elaborates  methods of quantum groups and  the
 Sz.-Nagy's unitary dilation theory. We note that all irreducible representations of $Pol(Mat_{m,n})_q$ are known only in the cases when either $m=1$, or $n=1$ or $m=n=2$.

In this paper we use the following standard notations: ${\mathbb R}$ is the set of real numbers, $\mathbb Z$ denotes the set of integers, $\mathbb Z_+=\{0,1,2,\ldots\}$,
$\mathbb N=\{1,2,\ldots\}$. All algebras are assumed to be unital one over the field of complex numbers $\mathbb C$ and $q\in (0,1)$. We write $Mat_{n}$ for the space of $n\times n$ complex matrices. By $\{e_n:n\in\mathbb Z_+\}$ we denote the standard basis in the Hilbert space $\ell^2(\mathbb Z_+)$.

\section{The $*$-algebra $Pol(\mathbb C^n)_q$}

 The $*$-algebra $Pol(\mathbb C^n)_q$ is a $*$-algebra generated by $z_j$, $j=1,\ldots,n$, subject to the relations:
\begin{eqnarray*}
z_jz_k&=&qz_kz_j, j<k,\\
z_j^*z_k&=&qz_kz_j^*, j\ne k,\\
z_j^*z_j&=&q^2z_jz_j^*+(1-q^2)(1-\sum_{k>j}z_kz_k^*).
\end{eqnarray*}
This $*$-algebra is a $q$-analog of the $*$-algebra of polynomials on $\mathbb C^n$. It  was first introduced by Pusz and Woronowicz in \cite{pusz-woronowicz} but in terms of slightly different generators (see a remark in \cite[section 3]{vaksman-boundary}). 
Its irreducible representations are well-known, see \cite{pusz-woronowicz}.
If $n=1$ we have the following list of representations up to unitary equivalence:
\begin{enumerate}
\item
the Fock representation $\rho_F$ on $\ell^2(\mathbb Z_+)$:
$\rho_F(z_1)e_n=\sqrt{1-q^{2n+2}}e_{n+1}$,
\item one-dimensional representations $\rho_\varphi$, $\varphi\in[0,2\pi)$: $\rho_{\varphi}(z_1)=e^{i\varphi}$.
\end{enumerate}

The $*$-algebra $\mathbb C[SU_2]_q$ of regular functions on the quantum $SU_2$  is given by its generators $t_{ij}$, $i,j=1,2$, satisfying the relations:
\begin{eqnarray}
&&t_{11}t_{21}=qt_{21}t_{11}, \ t_{11}t_{12}=qt_{12}t_{11},\  t_{12}t_{21}=t_{21}t_{12},\nonumber\\
&&t_{22}t_{21}=q^{-1}t_{21}t_{11},\  t_{22}t_{12}=q^{-1}t_{12}t_{22}, \label{r2}\\
&&t_{11}t_{22}-t_{22}t_{11}=(q-q^{-1})t_{12}t_{21},\  t_{11}t_{22}-qt_{12}t_{21}=1,\nonumber\\
&&t_{11}^*=t_{22}, t_{12}^*=-qt_{21}. \nonumber
\end{eqnarray}

By \cite{soibelman-vaksman} any irreducible representation of $\mathbb C[SU_2]_q$ is unitarily equivalent to one of the  following:
\begin{enumerate}
\item one-dimensional representations:
\begin{equation}
\xi_\varphi(t_{11})=e^{i\varphi}, \xi_\varphi(t_{21})=0, \varphi\in[0,2\pi).\label{onedim}
\end{equation}
\item infinite-dimensional representations $\pi_\varphi$, $\varphi\in[0, 2\pi)$, on $\ell^2({\mathbb Z}_+)$:
\begin{eqnarray}
&&\pi_\varphi(t_{11})e_0=0, \pi_\varphi(t_{11})e_k=(1-q^{2k})^{1/2}e_{k-1}, k\geq 1,\nonumber\\
&&\pi_{\varphi}(t_{21})e_k=q^{k}e^{i\varphi}e_k,\label{csu2}\\
&&\pi_\varphi(t_{22})e_k=(1-q^{2(k+1)})^{1/2}e_{k+1},\nonumber\\
&&\pi_{\varphi}(t_{12})e_k=-q^{k+1}e^{-i\varphi}e_k.\nonumber
\end{eqnarray}
\end{enumerate}

\section{The $*$-algebra $Pol(Mat_{2})_q$ and $*$-representations}\label{secpolm2q}
The $*$-algebra $Pol(Mat_{2})_q$, a  $q$-analog of
polynomials on the space  $Mat_{2}$  of complex $2\times 2$ matrices, introduced in \cite{sv}, is
given by its generators  $\{z_{a}^{\alpha}\}_{a=1,2;\alpha=1,2}$ and the
following commutation relations:
\begin{equation}
\begin{array}{rclrcl}\label{ff1}
z_1^1z_2^1&=&qz_2^1z_1^1,& z_2^1z_1^2&=&z_1^2z_2^1,\\
z_1^1z_1^2&=&qz_1^2z_1^1,& z_2^1z_2^2&=&qz_2^2z_2^1,\\
z_1^1z_2^2-z_2^2z_1^1&=&(q-q^{-1})z_1^2z_2^1,& z_1^2z_2^2&=&qz_2^2z_1^2,\\
\end{array}
\end{equation}
\begin{equation}
\begin{array}{rcl}\label{ff2}
(z_1^1)^*z_1^1&=&q^2z_1^1(z_1^1)^*-
(1-q^2)(z_2^1(z_2^1)^*+z_1^2(z_1^2)^*)+\\
&&\qquad\quad+q^{-2}(1-q^2)^2z_2^2(z_2^2)^*+1-q^2,\\
(z_2^1)^*z_2^1&=&q^2z_2^1(z_2^1)^*-(1-q^2)z_2^2(z_2^2)^*+1-q^2,\\
(z_1^2)^*z_1^2&=&q^2z_1^2(z_1^2)^*-(1-q^2)z_2^2(z_2^2)^*+1-q^2,\\
(z_2^2)^*z_2^2&=& q^2z_2^2(z_2^2)^*+1-q^2,
\end{array}
\end{equation}
%\begin{eqnarray}
%&\mbox{and}\nonumber\\
%and
\begin{eqnarray}
&\begin{array}{rclrcl}\label{ff3}
(z_1^1)^*z_2^1-qz_2^1(z_1^1)^*&=&(q-q^{-1})z_2^2(z_1^2)^*,&
(z_2^2)^*z_2^1&=&qz_2^1(z_2^2)^*,\\
(z_1^1)^*z_1^2-qz_1^2(z_1^1)^*&=&(q-q^{-1})z_2^2(z_2^1)^*,&
(z_2^2)^*z_1^2&=&qz_1^2(z_2^2)^*,\\
(z_1^1)^*z_2^2&=&z_2^2(z_1^1)^*,&
(z_2^1)^*z_1^2&=&z_1^2(z_2^1)^*.
\end{array}
\end{eqnarray}

The irreducible representations of $Pol(Mat_{2})_q$ were classified
in \cite{T}.
Next theorem presents them in a different form that is convenient for our purpose.

Let
$C,\ S,\ d(q)\colon \ell^2(\mathbb{Z}_+)\rightarrow \ell^2(\mathbb{Z}_+)$
be operators  defined as follows
\[
Se_n=e_{n+1},\ C e_n =(1-q^{2n})^{\frac{1}{2}}e_n,\ d(q)e_n=q^n e_n.
\]

\begin{theorem}
Any irreducible bounded representation of $Pol(Mat_2)_q$ is unitarily equivalent to one of the following  non-equivalent representations:

\begin{enumerate}
\item the Fock representation acting in $\ell^2(\mathbb{Z}_+)^{\otimes 4}$:
\begin{align*}
\pi_F(z_2^2)&=CS\otimes 1\otimes 1\otimes 1, \\
\pi_F(z_2^1)& =d(q)\otimes CS\otimes 1\otimes 1,\\
\pi_F(z_1^2)& = d(q)\otimes 1\otimes CS\otimes 1,\\
\pi_F(z_1^1)&=1\otimes d(q)\otimes d(q)\otimes CS-
q^{-1}S^*C\otimes CS\otimes CS\otimes 1;
\end{align*}
\iffalse
%The  $C^*$-algebra generated by the operators of the
%Fock representation will be denoted $\mathcal{A}_F$.
\fi
\item representations
$\tau_{\varphi}$, $\varphi\in[0,2\pi)$, acting in $\ell^2(\mathbb{Z}_+)^{\otimes 3}$:
\begin{align*}
\tau_{\varphi}(z_2^2)&=CS\otimes 1\otimes 1,\\
\tau_{\varphi}(z_2^1)& =d(q)\otimes CS\otimes 1,\\
\tau_{\varphi}(z_1^2)& = d(q)\otimes 1\otimes CS,\\
\tau_{\varphi}(z_1^1)&=e^{i\varphi}1\otimes d(q)\otimes d(q)-
q^{-1}S^*C\otimes CS\otimes CS;
\end{align*}
\iffalse
%The corresponding $C^*$-algebra will be denoted by
%$\mathcal{A}_{\tau,\varphi}$.
\fi
\item    representations
$\nu_{1,\varphi}$
and $\nu_{2,\varphi}$ acting in $\ell^2(\mathbb{Z}_+)^{\otimes 2}$:
\begin{itemize}
\item[3a)]
\begin{align*}
\nu_{1,\varphi}(z_2^2)&=CS\otimes 1,\\
\nu_{1,\varphi}(z_2^1)& =e^{i\varphi}d(q)\otimes 1,\\
\nu_{1,\varphi}(z_1^2)& = d(q)\otimes CS,\\
\nu_{1,\varphi}(z_1^1)&=-e^{i\varphi}q^{-1}S^*C\otimes CS,
\end{align*}
\item[3b)]
\begin{align*}
\nu_{2,\varphi}(z_2^2)&=CS\otimes 1,\\
\nu_{2,\varphi}(z_2^1)& = d(q)\otimes CS,\\
\nu_{2,\varphi}(z_1^2)& =e^{i\varphi}d(q)\otimes 1,\\
\nu_{2,\varphi}(z_1^1)&=-e^{i\varphi}q^{-1}S^*C\otimes CS;
\end{align*}
\end{itemize}
\iffalse
%The corresponding $C^*$-algebras are denoted by
%by $A_{\nu_1,\varphi}$ and $A_{\nu_2,\varphi}$ respectively.
\fi
\item  representations
$\rho_{\varphi_1,\varphi_2}$, $\varphi_1,\varphi_2\in[0,2\pi)$, acting in $\ell^2(\mathbb{Z}_+)$:
\begin{align*}
\rho_{\varphi_1,\varphi_{2}}(z_2^2)&=CS,\\
\rho_{\varphi_1,\varphi_2}(z_2^1)& =e^{i\varphi_1}d(q),\\
\rho_{\varphi_1,\varphi_2}(z_1^2)&=e^{i\varphi_2}d(q),\\
\rho_{\varphi_,\varphi_2}(z_1^1)&=
-e^{i(\varphi_1+\varphi_2)}q^{-1}S^*C;
\end{align*}
\iffalse
%The corresponding $C^*$-algebras are denoted by
%$\mathcal{A}_{\varphi_1,\varphi_2}$.
\fi
\item
 representations $\theta_{\varphi}$,
$\varphi\in [0,2\pi)$, acting in $l_2(\mathbb{Z}_+)$:
\[
\theta_{\varphi}(z_2^2)=e^{i\varphi},\ \theta_{\varphi}(z_2^1)=
\theta_{\varphi}(z_1^2)=0,\
\theta_{\varphi}(z_1^1)=q^{-1}CS;
\]
\iffalse
%The corresponding $C^*$-algebras are denoted by $\mathcal{A}_{\theta,\varphi}$
\fi
\item  one-dimensional representations $\gamma_{\varphi_1,\varphi_2}$, where
$\varphi_1,\varphi_2\in [0,2\pi)$
\[
\gamma_{\varphi_1,\varphi_2}(z_2^2)=e^{i\varphi_1},\ \gamma(z_2^1)=\gamma(z_1^2)=0,\
\gamma_{\varphi_1,\varphi_2}(z_1^1)=e^{i\varphi_2}q^{-1}.
\]
\iffalse
%The $C^*$-algebra generated by
%$\gamma_{\varphi_1,\varphi_2}$ is denoted by
%$\mathcal{B}_{\varphi_1,\varphi_2}=\mathbb{C}$.
\fi
\end{enumerate}
\end{theorem}

It can be easily seen from the above list of representations that the $*$-algebra $Pol(Mat_{2})_q$ is  $*$-bounded
(see \cite{OS}), i.e., there exist constants $C(a)$, $a\in Pol(Mat_{2})_q$,
such that $||\pi(a)||\leq C(a)$ for any bounded $*$-representation $\pi$. Let $C(\text{Mat}_2)_q$ denote the universal enveloping $C^*$-algebra of $Pol(Mat_2)_q$.
The following theorem was proved first in \cite{pro_tur} and in general case for $Pol(Mat_n)_q$ in \cite{ssv}.
\begin{theorem}\label{th1}
Given an irreducible representation, $\pi$, of $C(Mat_2)_q$,
 let $\mathcal{A}_{\pi}$ be the
$C^*$-algebra generated by operators of the representation $\pi$.
Then there exists a homomorphism $\delta_{\pi}$ from the $C^*$-algebra
$\mathcal{A}_{\pi_F}$ to the $C^*$-algebra $\mathcal{A}_{\pi}$ such that
\[
\delta_{\pi}(\pi_{F}(z_i^j))=\pi(z_i^j),\quad i,\ j=1,2.
\]
Consequently, the Fock representation $\pi_F$ of $C(Mat_2)_q$ is faithful and $C(Mat_2)_q\simeq {\mathcal A}_{\pi_F}$.
\end{theorem}

%In \cite{pro_tur} it was proved that the Fock
%representation $\pi_F$ of $Pol(Mat_2)_q$ is faithful and
%and for any other representation $\pi$ of $Pol(Mat_2)_q$, $||\pi(a)||\leq||\pi_F(a)||$ showing that there exists a universal enveloping $C^*$-algebra of the $*$-algebra $Pol(Mat_2)_q$ which is the completion of $Pol(Mat_2)_q$ with respect to the norm $||a||=||\pi_F(a)||$. Let  ${\mathbb C}(\text{Mat}_2)_q$ denote this universal enveloping $C^*$-algebra and ${\mathbb A}(Mat_2)_q$ the closure of the subalgebra (not $*$-subalgebra) generated $\{z_{a}^{\alpha}\}_{a=1,2;\alpha=1,2}$.
%${\mathbb A}[\text{Mat}_2]_q$ is a $q$-analogue of the algebra of holomorphic polynomials in
%the vector space $Mat_{2}$.

In what follows we  will use another description of irreducible representations. For this we need the following $*$-homomorphisms whose existence was indicated in \cite{bershtein} without a proof.
%defined by
%Vaksman, Sinelshchikov and Shklyarov, \cite{ssv}.

%It was shown in \cite{bershtein}  that the above list of irreducible representations can be obtained using
%the following $*$-homomorphisms:

\begin{lemma} The map
$$ \mathcal D:z_j^i\mapsto\sum_{a,b=1}^2 z_b^a\otimes t_{bj}\otimes t_{ai}, i,j=1,2,$$
is uniquely extendable to a $*$-homomorphism  $$\mathcal D:Pol(\text{Mat}_2)_q\to Pol(\text{Mat}_2)_q\otimes\mathbb C[SU_2]_q\otimes \mathbb C[SU_2]_q.$$
\end{lemma}
%it corresponds to the left-right action of $SU_2\times SU_2$ in $Mat_2$: $(X,Y):Z\to XZY^{-1}$, $X,Y\in SU_2$ $Z\in Mat_2$;
\begin{proof}
Consider the Hopf algebra $\mathbb C[SL_4]_q$ generated by $\{t_{ij}\}_{i,j=1}^4$ and the commutation relations
\begin{eqnarray*}
t_{ij}t_{kl}-qt_{kl}t_{ij}=0, \quad i=k\  \& \ j<l\text{ or } i<k\  \& \ j=l,\\
t_{ij}t_{kl}-t_{kl}t_{ij}=0\quad i<k\  \& \ j>l,\\
t_{ij}t_{kl}-t_{kl}t_{ij}=(q-q^{-1})t_{il}t_{kj} \quad i<k\  \& \ j<l,\\
\text{det}_qT:=\sum_{s\in S_4}(-q)^{l(s)}t_{1s(1)}t_{2s(2)}t_{3s(3)}t_{4s(4)}=1,
\end{eqnarray*}
with $l(s)=\text{card}\{(i,j):i<j\& s(i)>s(j)\}$.
 The comultiplication $\Delta$ is given by
$$\Delta(t_{ij})=\sum_{k=1}^2 t_{ik}\otimes t_{kj}.$$
Let $t=t_{\{1,2\}\{3,4\}}:=t_{13}t_{24}-qt_{24}t_{23}$.
Consider the map
$$\mathcal I: z_a^\alpha\mapsto t^{-1}t_{\{1,2\}J_{a\alpha}}.$$
where $J_{a\alpha}=\{3,4\}\setminus\{5-\alpha\}\cup\{a\}$ and $t_{IJ}:=t_{i_1j_1}t_{i_2j_2}-qt_{i_1,j_2}t_{i_2j_1}$ with $I=\{i_1,i_2\}$, $J=\{j_1,j_2\}$.
By \cite[Proposition 6.10]{ssv-dif}, it determines a homomorphism from $Pol(Mat_2)_q$ to a localisation of $Pol(\tilde X)_q:=(\mathbb C[SL_4]_q,*)$ (see \cite[(6.6)]{ssv-dif} for the involution $*$ on $\mathbb C[SL_4]_q$) with respect to a multiplicative system $tt^*$, $(tt^*)^2,\ldots$.

Consider now the two-sided ideal $J\subset\mathbb C[SL_4]_q$ generated by $t_{kl}$ with $k\leq 2$ and $l>2$ or $k>2$ and $l\leq 2$, and the canonical onto morphism
$$j:\mathbb C[SL_4]_q\to\mathbb C[SL_4]_q/J.$$
Let
$\mathbb C[S(U_2\times U_2)]_q= (\mathbb C[SL_4]_q/J,\star)$, an involutive algebra,  where
$$t_{ij}^\star=(-q)^{j-i}\text{det}_qT_{ij}$$
and $T_{ij}$ is derived from $(t_{ij})_{i,j=1}^4$ by deleting its $i$-th row and $j$-th column and
 $\text{det}_q T$ is the quantum determinant of $T$ (see \cite{ssv-dif} for the definition).
By \cite[Lemma 9.3]{ssv-dif} the composition $\tilde\Delta=(\text{id}\otimes j)\Delta$ is a homomorphsim of $*$-algebras $\tilde\Delta:Pol(\tilde X)_q\to Pol(\tilde X)_q\otimes\mathbb C[S(U_2\times U_2)]_q$. The $*$-homomorphism can be naturally extended to the localization of $Pol(\tilde X)_q$ which we shall also denote by $\tilde\Delta$.

Let $I$ be the two-sided ideal of $\mathbb C[S(U_2\times U_2)]_q$ generated by $t_{11}t_{22}-qt_{12}t_{21}-1$ and $t_{33}t_{44}-qt_{34}t_{43}-1$  (we note that
$1=\text{det}_q(t_{ij})=(t_{11}t_{22}-qt_{12}t_{21})(t_{33}t_{44}-qt_{34}t_{43})$ in $\mathbb C[S(U_2\times U_2)]_q$). Then $\mathbb C[SU_2]_q\otimes\mathbb C[SU_2]_q=\mathbb C[S(U_2\times U_2)]_q/I$.

  Let $$i:\mathbb C[S(U_2\times U_2)]_q\mapsto \mathbb C[SU_2]_q\otimes\mathbb C[SU_2]_q$$ be the canonical onto homomorphism.

Then $(\text{id}\otimes i)\circ\tilde\Delta\circ\mathcal I$ is a $*$-homomorphism from $Pol((Mat_2)_q$ to $Pol(Mat_2)_q\otimes\mathbb C[SU_2]_q\otimes \mathbb C[SU_2]_q$. To prove the lemma   it is enough to see now that
$\mathcal D=(\text{id}\otimes i)\circ\tilde\Delta\circ\mathcal I$.

Using relations in $\mathbb[SL_4]_q/J$ we obtain
%\end{proof}
%\end{document}
\begin{eqnarray*}
\tilde\Delta(t)&=&\tilde\Delta(t_{13}t_{24}-qt_{14}t_{23})=(\sum_{k=3}^4t_{1k}\otimes t_{k3})(\sum_{i=3}^4t_{2i}\otimes t_{i4})-q(\sum_{k=3}^4t_{1k}\otimes t_{k4})(\sum_{i=3}^4t_{2i}\otimes t_{i3})\\
&=&\sum_{k,i=3}^4t_{1k}t_{2i}\otimes(t_{k3}t_{i4}-qt_{k4}t_{i3})=t_{13}t_{24}\otimes(t_{33}t_{44}-qt_{34}t_{43})\\
&+&t_{14}t_{23}\otimes(t_{43}t_{34}-qt_{44}t_{33})\\
&=&t_{13}t_{24}\otimes(t_{33}t_{44}-qt_{34}t_{43})-q^{-1}t_{14}t_{23}\otimes(t_{44}t_{33}-q^{-1}t_{43}t_{34})\\
&=&(t_{13}t_{24}-q^{-1}t_{14}t_{23})\otimes(t_{33}t_{44}-qt_{34}t_{43})=t\otimes(t_{33}t_{44}-qt_{34}t_{43})
\end{eqnarray*}
and hence $(\text{id}\otimes i)\circ\tilde \Delta(t)=t\otimes 1\otimes 1$.

Similarly,
\begin{eqnarray*}
\tilde\Delta(t_{\{12\}J_{11}})&=&\tilde\Delta(t_{11}t_{23}-qt_{13}t_{21})=(\sum_{k=1}^2t_{1k}\otimes t_{k1})(\sum_{i=3}^4t_{2i}\otimes t_{i3})\\
&-&q(\sum_{k=3}^4t_{1k}\otimes t_{k3})(\sum_{i=1}^2t_{2i}\otimes t_{i1})\\
&=&\sum_{k=1}^2\sum_{i=3}^4t_{1k}t_{2i}\otimes t_{k1}t_{i3}-q\sum_{i=3}^4\sum_{k=1}^2t_{1i}t_{2k}\otimes t_{i3}t_{k1}\\
&=&\sum_{k=1}^2\sum_{i=3}^4(t_{1k}t_{2i}-qt_{1i}t_{2k})\otimes t_{k1}t_{i3}
\end{eqnarray*}

and

\begin{eqnarray*}
(\text{id}\otimes i)\circ\tilde\Delta(t^{-1}t_{\{12\}J_{11}})&=&\sum_{k=1}^2\sum_{i=3}^4(t^{-1}(t_{1k}t_{2i}-qt_{1i}t_{2k}))\otimes t_{k1}\otimes t_{(i-2)1}\\
&=&\sum_{k=1}^2\sum_{i=3}^4 \mathcal I(z_k^{i-2})\otimes t_{k1}\otimes t_{(i-2)1}
\end{eqnarray*}
giving $(\text{id}\otimes i)\circ\tilde\Delta\circ \mathcal I(z_1^1)=\mathcal D(z_1^1)$. Similarly one checks that $(\text{id}\otimes i)\circ\tilde\Delta\circ \mathcal I(z_i^j)=\mathcal D(z_i^j)$ for other generators $z_i^j$.

\end{proof}

 Consider now a mapping
$\Pi_\varphi:Pol(Mat_2)_q\to Pol(\mathbb C)_q$  given on the generators  by $$\left(\begin{array}{cc}\Pi_\varphi(z_1^1)&\Pi_\varphi(z_1^2)\\\Pi_\varphi(z_2^1)&\Pi_\varphi(z_2^2)\end{array}\right)\to\left(\begin{array}{cc}q^{-1}z&0\\0&e^{i\varphi}\end{array}\right).$$
It is straight forward to check that $\Pi_\varphi$ is a $*$-homomorphism.

Clearly, if $\rho$ is a $*$-representation of $Pol(\mathbb C)_q$, $\tau$ is a $*$-representation of $Pol(Mat_2)_q$, and $\pi_1$, $\pi_2$ are representation of
$\mathbb C[SU_2]_q$ then $\rho\circ\Pi_\varphi$ and  $(\tau\otimes\pi_1\otimes\pi_2)\circ\mathcal D$  are $*$-representations of $Pol(Mat_2)_q$.

Let $\rho_F$ be the Fock representation of $Pol(\mathbb C)_q$ , $\rho_\varphi$, $\varphi\in[0,2\pi)$,  be the one-dimensional representations of $Pol(\mathbb C)_q$ and $\pi_\varphi$, $\varphi\in [0,2\pi)$  be the infinite-dimensional representation of $\mathbb C[SU_2]_q$  given by (\ref{csu2}). Consider the following families of $*$-representations of $Pol(Mat_2)_q$:
\[
\mathcal F_\varphi=\rho_F\circ \Pi_\varphi,\quad \chi_{\varphi_1,\varphi_2}=\rho_{\varphi_1}\circ\Pi_{\varphi_2}
\]
and
$$(\mathcal F_\varphi\otimes\pi_0\otimes\pi_0)\circ\mathcal D, \quad (\chi_{\varphi_1,\varphi_2}\otimes\pi_0\otimes\pi_0)\circ\mathcal D,$$
where $\varphi,\varphi_1,\varphi_2\in[0,2\pi)$.

We have
\begin{eqnarray*}
(\mathcal F_\varphi\otimes\pi_0\otimes\pi_0)\circ\mathcal D(z_1^1)=(\mathcal F_\varphi\otimes\pi_0\otimes\pi_0)(\sum z_b^a\otimes t_{b1}\otimes t_{a1})\\=
q^{-1}\rho_F(z)\otimes\pi_0(t_{11})\otimes\pi_0(t_{11})+e^{i\varphi}\otimes\pi_0(t_{21})\otimes\pi_0(t_{21})\\=
q^{-1}CS\otimes S^*C\otimes S^*C+e^{i\varphi}\otimes d(q)\otimes d(q),\\ \\
(\mathcal F_\varphi\otimes\pi_0\otimes\pi_0)\circ\mathcal D(z_2^2)=(\mathcal F_\varphi\otimes\pi_0\otimes\pi_0)(\sum z_b^a\otimes t_{b2}\otimes t_{a2})\\=
q^{-1}\rho_F(z)\otimes\pi_0(t_{12})\otimes\pi_0(t_{12})+e^{i\varphi}\otimes\pi_0(t_{22})\otimes\pi_0(t_{22})\\=
qCS\otimes d(q)\otimes d(q)+e^{i\varphi}\otimes CS\otimes CS,\\ \\
(\mathcal F_\varphi\otimes\pi_0\otimes\pi_0)\circ\mathcal D(z_1^2)=(\mathcal F_\varphi\otimes\pi_0\otimes\pi_0)(\sum z_b^a\otimes t_{b1}\otimes t_{a2})\\=
q^{-1}\rho_F(z)\otimes\pi_0(t_{11})\otimes\pi_0(t_{12})+e^{i\varphi}\otimes\pi_0(t_{21})\otimes\pi_0(t_{22})\\=
-CS\otimes S^*C\otimes d(q)+e^{i\varphi}\otimes d(q)\otimes CS,\\ \\
(\mathcal F_\varphi\otimes\pi_0\otimes\pi_0)\circ\mathcal D(z_2^1)=(\mathcal F_\varphi\otimes\pi_0\otimes\pi_0)(\sum z_b^a\otimes t_{b2}\otimes t_{a1})\\=
q^{-1}\rho_F(z)\otimes\pi_0(t_{12})\otimes\pi_0(t_{11})+e^{i\varphi}\otimes\pi_0(t_{22})\otimes\pi_0(t_{21})\\=
-CS\otimes d(q)\otimes S^*C+e^{i\varphi}\otimes CS\otimes d(q),
\end{eqnarray*}
  and
\begin{eqnarray*}
(\chi_{\varphi_1,\varphi_2}\otimes\pi_0\otimes\pi_0)\circ\mathcal D(z_1^1)&=&q^{-1}e^{i\varphi_1} S^*C\otimes S^*C+e^{i\varphi_2} d(q)\otimes d(q),\\
(\chi_{\varphi_1,\varphi_2}\otimes\pi_0\otimes\pi_0)\circ\mathcal D(z_2^2)&=&qe^{i\varphi_1}d(q)\otimes d(q)+e^{i\varphi_2} CS\otimes CS,\\
(\chi_{\varphi_1,\varphi_2}\otimes\pi_0\otimes\pi_0)\circ\mathcal D(z_1^2)&=&-e^{i\varphi_1} S^*C\otimes d(q)+e^{i\varphi_2}d(q)\otimes CS,\\
(\chi_{\varphi_1,\varphi_2}\otimes\pi_0\otimes\pi_0)\circ\mathcal D(z_2^1)&=&-e^{i\varphi_1}d(q)\otimes S^*C+e^{i\varphi_2} CS\otimes d(q).
\end{eqnarray*}

\begin{lemma}\label{ber}
The $*$-representation $(\mathcal F_\varphi\otimes\pi_0\otimes\pi_0)\circ\mathcal D$, $\varphi\in[0,2\pi)$, is irreducible and unitarily equivalent to $\tau_\varphi$.
\iffalse
The $*$-representations $(\chi_{\varphi_1,\varphi_2}\otimes\pi_0\otimes\pi_0)\circ\mathcal D$ are reducible whose irreducible subrepresentations are  unitarily equivalent to some $\rho_{\psi_1,\psi_2}$.
\fi
\end{lemma}
\begin{proof}
Fix $\varphi\in[0,2\pi)$.
\iffalse
Show that  representations $\tau_{\varphi}$ and
$(\mathcal{F}_{\varphi}\otimes\pi_0\otimes\pi_0)\circ\mathcal{D}$ are
irreducible and unitarily equivalent.
\fi
Let $Z_i^j=\tau_{\varphi}(z_i^j)$ and
$W_i^j=(\mathcal{F}_{\varphi}\otimes\pi_0\otimes\pi_0)\circ\mathcal{D}
(z_i^j)$, $i,j=1,2$. The operators act on $\ell^2(\mathbb Z_+)^{\otimes 3}$.
Let $\Omega=e_0\otimes e_0\otimes e_0$. It can be easily verified that
$\Omega$ is cyclic for both of the families $\{Z_i^j,\ (Z_i^j)^*,\ i,j=1,2\}$,
$\{W_i^j,\ (W_i^j)^*,\ i,j=1,2\}$, and
\begin{gather*}
(Z_2^2)^* \Omega=(W_2^2)^* \Omega=0,\quad (Z_2^1)^* \Omega=(W_2^1)^*
\Omega=0,\\
(Z_1^2)^* \Omega=(W_1^2)^* \Omega=0,\quad
(Z_1^1)^* \Omega=(W_1^1)^* \Omega= e^{-i\phi} \Omega.
\end{gather*}
Hence both $\tau_{\varphi}$ and
$(\mathcal{F}_{\varphi}\otimes\pi_0\otimes\pi_0)\circ\mathcal{D}$
determine so-called {\bf coherent representations} of the Wick algebra corresponding to
$Pol(Mat_{2})_q$, with the same coherent state (see \cite{jsw} for definition and properties of coherent
representation of $*$-algebra allowing Wick ordering). Since coherent
representation of Wick algebra is unique, up to the unitary equivalence,
and irreducible (see \cite[Proposiiton 1.3.3]{jsw}), we have the required statement.

%SOME CONSIDERATIONS:

%1. In fact, for our purposes it will be enough to show that   $(\chi_{\varphi_1,\varphi_2}\otimes\pi_0\otimes\pi_0)\circ\mathcal D$
%annihilate Shilov ideal. Then it is direct sum of $\rho_{\phi_1,\phi_2}$ or $\gamma_{\phi_1,\phi_2}$.

%{\bf NOTE} There was a mistake in formula for $\gamma_{\phi_1,\phi_2}(z_1^1)$. The correct
%formula is $\gamma_{\phi_1,\phi_2}(z_1^1)=e^{i\phi_2}q^{-1}$ instead of
%$\gamma_{\phi_1,\phi_2}(z_1^1)=e^{i\phi_2}q^{-1}(1-q^2)^{-1/2}$ (I made corresponding correction in Theorem 1).

%3. As for equivalence $\mathcal{F}_{\phi}\otimes\pi_0\otimes\pi_0$ and $\tau_{\phi}$. It is sufficent to check that $\mathcal{F}_{\phi}\otimes\pi_0\otimes\pi_0$ is irreducible and
%the fact that any of $$\mathcal{F}_{\phi}\otimes\pi_0\otimes\pi_0 (z_{a}^b),\quad a\ne b$$
%is not normal. Then it is remained to show that $\mathcal{F}_{\phi}\otimes\pi_0\otimes\pi_0$ is not equivalent to $\pi_F$ (but for our goals the case %$\mathcal{F}_{\phi}\otimes\pi_0\otimes\pi_0\simeq \pi_F$ is Ok...)
\end{proof}

\section{Shilov boundary}

Let $E_1$ and $E_2$  be subspaces of $C^*$-algebras $\mathcal A_1$ and $\mathcal A_2$ respectively. We denote by $M_n(E_i)$ be the space of all $n\times n$ matrices with entries in $E_i$. We equip $M_n(E_i)$ with norms induced from the $C^*$-algebras $M_n(\mathcal A_i)$. Note that the norms are independent of the embeddings of $E_i$ into a $C^*$-algebra.
 Let $T:E_1\to E_2$ be a linear operator. Denote by $T^{(n)}$ the mapping from $M_n(E_1)$ to $M_n(E_2)$ defined by
  $$T^{(n)}((a_{ij})_{i,j})=(T(a_{ij})_{i,j}), (a_{ij})_{i,j}\in M_n(E_1).$$ $T$ is called contractive if $||T||\leq 1$ and completely contractive if $||T^{(n)}||\leq 1$ for any $n\geq 1$.
  $T$ is called an isometry if $||T(a)||_{E_2}=||a||_{E_1}$, and is a complete isometry if $T^{(n)}$ is an isometry for any $n\geq 1$.

Let $A$ be a linear subspace of a $C^*$-algebra $B$ such that $A$ contains the identity of $B$ and generates $B$ as a $C^*$-algebra.
The following definition was given by Arveson \cite{arveson}.

\begin{definition}
A closed two-sided ideal $J$ in $B$ is called a {\it boundary ideal} for $A$ if the canonical quotient map $q:B\to B/J$ is completely isometric on $A$.  A boundary ideal is called the {\it Shilov boundary} for $A$ if it contains every other boundary ideal.
\end{definition}

Note that the Shilov boundary exists and unique, \cite{arveson,hamana}.
Shilov boundary ideal is a non-commutative analog of Shilov boundary of a compact Hausdorff space $X$ relative to a subspace $\mathcal A$ of the space $C(X)$ of continuous functions on $X$ which is by definition the smallest closed subset $K$ of $X$ such that every function in $\mathcal A$ achieves its maximum modulus on
 $K$.

Let us give some examples of Shilov boundary and Shilov boundary ideals.
\begin{example}\rm
\begin{itemize}
\item If $\mathbb D=\{z\in \mathbb C^n: |z|\leq 1\}$ is the unit disk. It is known that any holomorphic function on $\mathbb D$ attains its maximum on the unit disk
$\mathbb U=\{z\in\mathbb C^n:|z|=1\}$ and moreover it is the smallest closed set with this property and hence $\mathbb U$ is the Shilov boundary of $\mathbb D$ with respect to the set of holomorphic functions $A(\mathbb D)$. The ideal $J=\{f\in C(\mathbb D): f|_{\mathbb U}=0\}$  is the Shilov ideal of the $C^*$-algebra $C(\mathbb D)$ with respect to
$A(\mathbb D)$.
\item A q-analog of $C(\mathbb D)$ is the universal enveloping $C^*$-algebra of $Pol(Mat_{n,1})_q$. It was proved by L.Vaksman, \cite{vaksman-boundary} that a closed two-sided  ideal generated by $\sum_{j=1}^n z_jz_j^*-1$ is the Shilov boundary ideal for the closed unital algebra generated by $z_i$, $i=1,\ldots,n$, which is a $q$-analog of the algebra of  holomorphic functions on
    $\mathbb D$.
\end{itemize}
\end{example}

\medskip

In $C(Mat_n)_q$ consider a closed two-sided ideal $J$ generated by $$\sum_{j=1}^n q^{2n-\alpha-\beta}z_j^\alpha(z_j^\beta)^*-\delta^{\alpha\beta}, \alpha,\beta=1,\ldots n,$$ where $\delta^{\alpha\beta}$ is the Kronecker symbol.
The ideal $J$ is a $*$-ideal, i.e. $J=J^*$. The quotient algebra $\mathbb C(S(\mathbb D))_q:=C(Mat_n)_q/J$ is a $U_qsu_{n,n}$-module $*$-algebra called the algebra of continuous functions on the Shilov boundary of a quantum matrix ball.
The canonical homomorphism
$$j_q:C(Mat_n)_q\to C(S(\mathbb D))_q$$
is a $q$-analog of the restriction operator which maps a continuous functions on the disk $\mathbb D=\{z\in Mat_n:zz^*\leq 1\}$ to its restriction to the Shilov boundary $S(\mathbb D)=\{z\in Mat_n: zz^*=1\}$.

In this section we show that for $n=2$, the ideal $J$ is the Shilov boundary ideal for the (non-involutive) closed subalgebra $A(Mat_2)_q$ of $C(Mat_2)_q$ generated by $z_i^j$, $i,j=1,2$.  Our approach, similarly to \cite{vaksman-boundary}, is based on Sz.-Nagy's and Foyas' dilation theory.

\begin{theorem}
{\bf (Sz.-Nagy's dilation theorem).} Let $T\in B(H)$ with $||T||\leq 1$. Then there exists a Hilbert space $K$ containing $H$ as a subspace and a unitary $U$ on $K$ with the property that
$$T^n=P_HU^n|_H \text{ for all nonnegative integers } n.$$
\end{theorem}

\begin{lemma}\label{theonly}
The only irreducible representations that annihilate the ideal $J$ are $\rho_{\varphi_1,\varphi_2}$ and $\gamma_{\varphi_1,\varphi_2}$, $\varphi_1$, $\varphi_2\in[0,2\pi)$.
\end{lemma}
\begin{proof} A straightforward verification.
\end{proof}

\begin{lemma}\label{inequality}
Given a representation $\pi$  of $Pol(Mat_2)_q$ that annihilates the ideal $J$ and  $a\in Pol (Mat_2)_q$, $||\pi(a)||\le \sup_{\psi_1,\psi_2} ||\rho_{\psi_1,\psi_2}(a)||$.
\end{lemma}
\begin{proof}
We start by noting that the operators $C$, $S$, $d(q)\colon \ell^2(\mathbb{Z}_{+})\rightarrow \ell^2(\mathbb{Z}_{+})$ defined in Section \ref{secpolm2q} satisfy the equalities
\begin{equation}\label{Cformula}
C^2=(1-q^2)\sum_{n=0}^{\infty} q^{2n} S^{n+1} (S^{n+1})^*,\quad d(q)=\sum_{n=0}^\infty q^n \left(S^n(S^n)^*-S^{n+1}(S^{n+1})^*\right)
\end{equation}
and hence the $C^*$-algebra, $C^*(S)$, generated by $S$ coincides with the one generated by $S$, $C$ and $d(q))$, and the mapping $S\mapsto e^{i\varphi}$ can be naturally extended to $*$-homomorphism
\[
\Theta_{\varphi}\colon C^*(S)\rightarrow\mathbb{C},\quad \Theta_{\varphi}(S)=e^{i\varphi},\quad \Theta_{\varphi}(C)=1,\quad\Theta_{\varphi} (d(q))=0.
\]
Recall that
\begin{eqnarray}
\rho_{\varphi_1,\varphi_{2}}(z_2^2)&=&CS,\nonumber\\
\rho_{\varphi_1,\varphi_2}(z_2^1) &=&e^{i\varphi_1}d(q),\label{rho}\\
\rho_{\varphi_1,\varphi_2}(z_1^2)&=&e^{i\varphi_2}d(q),\nonumber\\
\rho_{\varphi_,\varphi_2}(z_1^1)&=&
-e^{i(\varphi_1+\varphi_2)}q^{-1}S^*C,\nonumber
\end{eqnarray}
and
\[
\gamma_{\varphi_1,\varphi_2}(z_2^2)=e^{i\varphi_1},\ \gamma_{\varphi_1,\varphi_2}(z_2^1)=\gamma_{\varphi_1,\varphi_2}(z_1^2)=0,\
\gamma_{\varphi_1,\varphi_2}(z_1^1)=e^{i\varphi_2}q^{-1}.
\]

 For a representation $\pi$ of $Pol(Mat_2)_q)$ we let  $\mathcal A_\pi$ denote  the unital $C^*$-algebra generated by $\pi(z_i^j)$, $i,j=1,2$. Then
$\mathcal{A}_{\rho_{\varphi_1,\varphi_2}}= C^*(S)$. In fact it follows from $(\ref{rho})$ and (\ref{Cformula}) that $\mathcal{A}_{\rho_{\varphi_1,\varphi_2}}\subset C^*(S)$.
 To see the other inclusion we note that $0$ is an isolated point in the spectrum $\sigma(C)$ of $C$, and hence the function $f$ given by $f(0)=0$  and $f(t)=t^{-1}$, $t\in\sigma(C)$,
  $t\ne 0$, is continuous on $\sigma(C)$.
  Therefore, since  $T:=\rho_{\varphi_1,\varphi_{2}}(z_2^2)=CS$ one has $C=((1-q^{-2})I+q^{-2}T^*T)^{1/2}\in \mathcal{A}_{\rho_{\varphi_1,\varphi_2}}$ and
$S=f(C)T\in \mathcal{A}_{\rho_{\varphi_1,\varphi_2}}$ implying  $C^*(S)\subset \mathcal{A}_{\rho_{\varphi_1,\varphi_2}}$.

Evidently, $\mathcal{B}_{\gamma_{\varphi_1,\varphi_2}}=\mathbb{C}$.
The homomorphism $\Theta_{\varphi_1}$ gives rise to a homomorphism between $\mathcal{A}_{\rho_{\varphi_1,\pi+\varphi_2}}$ and $\mathcal{B}_{\varphi_1,\varphi_2}$:
\[
\Theta_{\varphi_1}(\rho_{\varphi_1,\pi+\varphi_2}(z_i^j))=\gamma_{\varphi_1,\varphi_2}(z_i^j),\quad i,j=1,2
\]
proving that
\[
|\gamma_{\varphi_1,\varphi_2}(a)|=|\Theta_{\varphi_1}(\rho_{\varphi_1,\pi+\varphi_2}(a))|\leq\|\rho_{\varphi_1,\pi+\varphi_2}(a)\|\le
\sup_{\psi_1,\psi_2}||\rho_{\psi_1,\psi_2}(a)||,
\quad i=1,2.
\]
\end{proof}
\begin{lemma}
The ideal $J$ is a boundary ideal, i.e. the restriction $j_{A(Mat_2)_q}$ of $j_q$ to $A(Mat_2)_q$ is a complete isometry.
\end{lemma}
\begin{proof}
Since $j_q$ is a $*$-homomorphism between $C^*$-algebras, $j_q$ and hence $j_{A(Mat_2)_q}$ is a complete contraction.
Therefore it is enough to prove that for $a_{ij}\in  A(Mat_2)_q$, we have
$$||(\pi_F(a_{ij}))_{i,j}||_{M_n(C(Mat_2)_q)}\leq ||j_q^{(n)}((\pi_F(a_{ij}))||_{M_n(C(S(\mathbb D))_q)}.$$
Since by Lemma \ref{theonly} the only representations of $C(Mat_2)_q$ that annihilate the ideal $J$ are $\rho_{\varphi_1,\varphi_2}$ and $\gamma_{\varphi_1,\varphi_2}$,
$\varphi_i\in [0,2\pi)$, and
$$|\gamma_{\varphi_1,\varphi_2}(a)|\leq \sup_{\psi_1,\psi_2}||\rho_{\psi_1,\psi_2}(a)||$$
we have $$||b+J||_{C(S(\mathbb D))_q}=\sup_{\psi_1,\psi_2}||\rho_{\psi_1,\psi_2}(b)||.$$

Therefore, we must show that
$$||(\pi_F(a_{ij}))_{i,j}||_{M_n(C(Mat_2)_q)}\leq \sup_{\psi_1,\psi_2}||(\rho_{\psi_1,\psi_2}(a_{ij}))||_{M_n(B(\ell^2(\mathbb Z_+)\otimes\ell^2 (\mathbb Z_+)))}$$
for all $(a_{ij})\in M_n(A(Mat_2)_q)$. We will do this in two steps.

{\bf Step 1.} It follows from the definition of operators $C$ and $S$ that $T=CS$ is a contraction on $H=\ell^2(\mathbb Z_+)$. By Sz.-Nagy dilation theorem there exists a unitary operator $U$ on a Hilbert space $K$ with $K\supset H$ such that
$(CS)^n=P_HU^n|_H$ for any $n=1,2,\ldots$.  Consider a mapping $\Psi:\{z_i^j,i,j=1,2\}\to B(H^{\otimes 4})$  given by
$$\Psi(z_i^j)=\pi_F(z_i^j), (i,j)\ne(1,1),\text{ and }  \Psi(z_1^1)=1\otimes d(q)\otimes d(q)\otimes U-q^{-1}S^*C\otimes CS\otimes CS\otimes 1.$$
Then $\Psi$ extends uniquely to a homomorphism of $A(Mat_2)_q$ and $$\pi_F(a)=(1_{H^{\otimes 3}}\otimes P_H)\Psi(a)|_{H^{\otimes 3}\otimes H}, \ a\in A(Mat_2)_q.$$ Moreover, it is easy to see that
$\Psi$ has an extension to a $*$-representation  of $Pol(Mat_2)_q$ whose irreducible subrepresentations are unitarily equivalent to $\tau_\varphi$, $\varphi\in[0,2\pi)$.
Therefore
$$||\pi_F(a)||\leq ||\Psi(a)||\leq\sup_{\varphi\in[0,2\pi)}||\tau_\varphi(a)||, a\in A(Mat_2)_q.$$
Similarly,
$$||(\pi_F(a_{ij}))||_{M_n(B(H^{\otimes 4}))}\leq\sup_{\varphi\in[0,2\pi)}||(\tau_\varphi(a_{ij}))||_{M_n(B(H^{\otimes 3}))}, (a_{ij})\in M_n(A(Mat_2)_q).$$

 {\bf Step 2.} Our next goal is to prove  that for any $\varphi\in[0,2\pi)$
 $$||\tau_\varphi(a)||\leq\sup_{\varphi_1,\varphi_2}||\rho_{\varphi_1,\varphi_2}(a)||, a\in A(Mat_2)_q.$$

It is a routine to verify that the representations $(\chi_{\varphi_1,\varphi_2}\otimes\pi_0\otimes\pi_0)\circ\mathcal D$ annihilate the ideal $J$ for any $\varphi_1,\varphi_2\in [0,2\pi)$. In particular this fact implies
\[
\sup_{\varphi_1,\varphi_2}||(\chi_{\varphi_1,\varphi_2}\otimes\pi_0\otimes\pi_0)\circ\mathcal D(z_i^j)||\le \sup_{\varphi_1,\varphi_2}||\rho_{\varphi_1,\varphi_2}(z_i^j)||,\ i,j=1,2.
\]
So, it will be enough for us to show that for any $\varphi\in [0,2\pi)$ and $i,j=1,2$
\[
||\tau_{\varphi}(z_i^j)||\le\sup_{\varphi_1,\varphi_2}
||(\chi_{\varphi_1,\varphi_2}\otimes\pi_0\otimes\pi_0)\circ\mathcal D(z_i^j)||.
\]
Indeed, as in the first step let $U$ be a unitary operator on a Hilbert space $K$ with $K\supset H$, $H=\ell^2(\mathbb Z_+)$, such that
 $$(CS)^n=P_HU^n|_H.$$
 Then the mapping $\Psi$ defined on the generators $z_i^j$, $i,j=1,2$, that replaces $CS$ by $U$ in the first component in the expressions for $\tau_{\varphi}(z_i^j)=\mathcal F_\varphi\otimes\pi_0\otimes\pi_0(\mathcal D(z_i^j))$ can be extended to a $*$-representation $\Psi$ of $Pol(Mat_2)_q$ whose all irreducible subrepresentations are unitarily equivalent to $(\chi_{\varphi_1,\varphi}\otimes\pi_0\otimes\pi_0)\circ\mathcal D$, and, moreover,
 $$ \mathcal F_\varphi\otimes\pi_0\otimes\pi_0(\mathcal D(a))=(P_H\otimes 1_{H^{\otimes 2}})\Psi(a)|_{H\otimes H^{\otimes 2}}, a\in A(Mat_2)_q$$

 Hence for $a\in A(Mat_2)_q$
 \begin{eqnarray*}
 ||\tau_\varphi(a)||&=&||(\mathcal F_\varphi\otimes\pi_0\otimes\pi_0)(\mathcal D)(a))||\leq||\Psi(a)||\\
 &\leq&
 \sup_{\varphi_1\in[0,2\pi)}|| (\chi_{\varphi_1,\varphi}\otimes\pi_0\otimes\pi_0)(\mathcal D(a))||\\
 &\leq&\sup_{\varphi_1,\varphi_2}||\rho_{\varphi_1,\varphi_2}(a)||,
 \end{eqnarray*}
the last inequality is due to Lemma \ref{inequality}. Using similar arguments one gets that
$$||\tau_\varphi^{(n)}((a_{ij}))||_{M_n(B(H^{\otimes 3}))}\leq\sup_{\varphi_1,\varphi_2}||(\rho_{\varphi_1,\varphi_2}(a_{ij}))||_{M_n(B(H^{\otimes 2}))}, (a_{ij})\in M_n(A(Mat_2)_q).$$
Combining the results from Step 1 and Step 2 we obtain
 $$||\pi_F^{(n)}((a_{ij}))||_{M_n(B(H^{\otimes 4}))}\leq\sup_{\varphi_1,\varphi_2}||(\rho_{\varphi_1,\varphi_2}(a_{ij}))||_{M_n(B(H^{\otimes 2}))} \text{ for all }(a_{ij})\in M_n(A(Mat_2)_q),$$
 giving the statement of the theorem.

\end{proof}

\begin{remark} \rm We have proved that for any $a\in A(Mat_2)_q$, $\pi_F(a)=P_H\psi(a)|_H$, where $\psi$ is a $*$-representation of $Pol(Mat_2)_q$ that annihilates the ideal $J$.
\end{remark}

\begin{theorem}
The ideal $J$ is the Shilov boundary ideal for the subalgebra $A(Mat_2)_q$.
\end{theorem}

\begin{proof}
Assume that $I$ is a boundary ideal for $ A(Mat_2)_q$ with $I\supset J$. We have, in particular, that the quotient maps $j_q: C(Mat_2)_q\to C(Mat_2)_q/J=C(S(\mathbb D))_q$  and
$i_q: C(Mat_2)_q\to C(Mat_2)_q/I=(C(Mat_2)_q/J)/(I/J)=C(S(\mathbb D))_q/(I/J)$ are isometries when restricted to $ A(Mat_2)_q$.
Therefore for
 $a\in A(Mat_2)_q$ we have $$\|a+J\|=\|a\|=\|(a+J)+I/J\|$$
  and hence
the quotient map $k_q:C(S(\mathbb D))_q\to C(S(\mathbb D))_q/(I/J)$ is an isometry when restricted to $A(Mat_2)_q+J$. In particular,
$0\ne ||z_i^j||=||z_i^j+J||=||(z_i^j+J)+I/J||$, $i\ne j$.

If $T$ is an irreducible representation of $C(S(\mathbb D))_q/(I/J)$ such that $T((z_i^j+J)+I/J)\ne 0$ then the representation $T\circ k_q$ is an irreducible representation of
$C(Mat_2)_q/J$ which does not vanish on $z_i^j+J$, $i\ne j$, and $T\circ k_q(I/J)=0$.
 The only irreducible representations of $C(S(\mathbb D))_q$ that do not vanish on $z_i^j+J$, $i\ne j$, are $\tilde\rho_{\varphi_1,\varphi_2}(a+J):=\rho_{\varphi_1,\varphi_2}(a)$. Therefore, $T\circ k_q$ is unitarily equivalent to one of $\tilde\rho_{\varphi_1,\varphi_2}$ and hence $T\circ k_q(I/J)=0$ implies $I/J\subset \ker\tilde\rho_{\varphi_1,\varphi_2}$.
Let
$$K=\{(\varphi_1,\varphi_2)\in [0,2\pi)\times[0,2\pi): \rho_{\varphi_1,\varphi_2}(I)=0\} \text{ and }X_K=\{(e^{i\varphi_1},e^{i\varphi_2}):(\varphi_1,\varphi_2)\in K\}.$$
We want to see that $K$ is dense in $[0,2\pi)\times[0,2\pi)$.

In $C(S(\mathbb D))_q$ consider the subalgebra generated by $z_i^j+J$, $i\ne j$. It is easily seen that the algebra is commutative and that the elements $z_i^j+J$, $i\ne j$ are normal in $C(S(\mathbb D))_q$, i.e. they commute  with their adjoints. This follows from the fact that the  operators
$\rho_{\varphi_1,\varphi_2}(z_i^j)$, $i\ne j$  commute and are normal for any $\varphi_1,\varphi_2\in [0,2\pi)$. The joint spectrum of
$\{\tilde\rho_{\varphi_1,\varphi_2}(z_1^2+J),  \tilde\rho_{\varphi_1,\varphi_2}(z_2^1+J) \}$ is $\{(e^{i\varphi_1}q^k, e^{i\varphi_2}q^k): k=0,1,\ldots\}\cup\{(0,0)\}$.
If  $T$ is  an irreducible representation of $C(S(\mathbb D))_q/(I/J)$ then it follows
from the description of the representations $\rho_{\varphi_1,\varphi_2}$ that if $(e^{i\varphi_1}, e^{i\varphi_2})$ is in the joint spectrum of $\{T\circ k_q(z_1^2+J),  T\circ k_q(z_2^1+J) \}$ then $T\circ k_q$ is unitarily equivalent to $\tilde\rho_{\varphi_1,\varphi_2}$ and hence $(\varphi_1,\varphi_2)\in K$.

Now, given a holomorphic function on $\mathbb D^2=\{(\xi_1,\xi_2)\in\mathbb C^2: |\xi_1|<1, |\xi_2|<1\}$ which is also continuous on $\overline{\mathbb D^2}$ we have
$$\|f(z_1^2+J,z_2^1+J)\|=\|f((z_1^2+J)+I/J,(z_2^1+J)+I/J\|.$$
As
\begin{eqnarray*}
\|f(z_1^2+J,z_2^1+J)\|&=&\sup_{(\varphi_1,\varphi_2)\in [0,2\pi)^2}\|f(\rho_{\varphi_1,\varphi_2}(z_1^2), \rho_{\varphi_1,\varphi_2}(z_2^1))\|\\
&=&\sup\{|f(\xi_1,\xi_2)|:(\xi_1,\xi_2)\in \cup_{k\geq 0} q^k\mathbb T^2\}=\sup_{(\xi_1,\xi_2)\in \mathbb T^2}|f(\xi_1,\xi_2)|\\
&=&\sup_{(\xi_1,\xi_2)\in \mathbb D^2}|f(\xi_1,\xi_2)|,
\end{eqnarray*}
(here $\mathbb T^2=\{(\xi_1,\xi_2)\in \mathbb C^2:|\xi_1|=|\xi_2|=1\}$ and the last two equalities follows from the maximum principle), and
\begin{eqnarray*}
&&\|f((z_1^2+J)+I/J,(z_2^1+J)+I/J)\|\\
&=&\sup\{f(T\circ k_q(z_1^2+J), T\circ k_q(z_2^1+J)): T\in \text{Irrep}( C(S(\mathbb D))_q/(I/J))\}\\
&=&\text{max}\{\sup_{(\varphi_1,\varphi_2)\in K}\|f(\rho_{\varphi_1,\varphi_2}(z_1^2), \rho_{\varphi_1,\varphi_2}(z_2^1))\|, f(0,0)\}\\
&=&\sup\{|f(\xi_1,\xi_2)|:(\xi_1,\xi_2)\in \cup_{k\geq 0} q^kX_K\}
\end{eqnarray*}
($\text{Irrep}(A)$ denote the set of all irreducible representations of $A$),
we obtain $$\sup\{|f(\xi_1,\xi_2)|:(\xi_1,\xi_2)\in \mathbb D^2\}=\sup\{|f(\xi_1,\xi_2)|:(\xi_1,\xi_2)\in \cup_{k\geq 0} q^kX_K\}.$$
Hence $\overline{\cup_{k\geq 0} q^kX_K}$ contains the Shilov boundary of $\mathbb D^2$ which is $\mathbb T^2$. Therefore $\mathbb T^2\supset \overline X_K\supset \mathbb T^2$ giving that $K$ is dense in $[0,2\pi)\times [0,2\pi)$.

%If $T$ is a faithful representation of $C(S(\mathbb D))_q/(I/J)$ then the joint spectrum of $\{T\circ k_q(z_1^2+J),  T\circ k_q(z_2^1+J) \}$ is $X_T=\cup_{(\varphi_1,\varphi_2)\in K_T}\{(e^{i\varphi_1}q^k, e^{i\varphi_2}q^k): k=0,1,\ldots, \}\cup{(0,0)}$ for some $K_T\subset[0,2\pi)\times [0,2\pi)$.

%Assume that $K_T$ is not dense in  $[0,2\pi)\times [0,2\pi)$. Then $X_T$ is not dense in $\cup_kq^k(\mathbb T\times\mathbb T)$.  In fact, if $(a,b)=\lim q^{k_n}(e^{i\varphi_n}, e^{i\psi_n})$ and $(a,b)\ne (0,0)$ we can find a subsequence $\{n_l\}$ such that $k_{n_l}=const=:K$  and hence $q^{-K}(a,b)=\lim_l (e^{i\varphi_{n_l}}, e^{i\psi_{n_l}})$ and if this holds for any $(a,b)\in\cup_kq^k(\mathbb T\times\mathbb T)$, we obtain that $K_T$ is dense in $[0,2\pi)\times [0,2\pi)$, a contradiction.
 %Therefore, there exists a continuous function $f\ne 0$ on $\cup_kq^k(\mathbb T\times\mathbb T)$ that vanishes on $X_T$. This implies
 %$$0\ne ||f(z_1^2+J,z_2^1+J)||=||f(T\circ k_q(z_1^2+J),  T\circ k_q(z_2^1+J))||=0.$$
 %A contradiction, showing that $K_T$ is dense and
 This implies $$I/J\subset \cap_{(\varphi_1,\varphi_2)\in K} \ker \tilde\rho_{\varphi_1,\varphi_2}=\{0\}$$
and $I=J$.
\end{proof}

\section*{Acknowledgements} The work on this paper was supported by the Swedish Institute, Visby Program. The paper was initiated during the visit of  D. Proskurin  to the Department of Mathematical Sciences at Chalmers University of Technology,  the warm hospitality and stimulating atmosphere are gratefully acknowledged.


\begin{thebibliography}{99}
\bibitem{arveson}{\sc W.Arveson}, {\it  Subalgebras of $C^*$-algebras}, {\rm Acta Math. 123 (1969) 122--141.}

\bibitem{arveson2}{\sc W.Arveson}, {\it  Subalgebras of $C^*$-algebras II}, {\rm Acta Math. 128 (1972) 271--308.}

\bibitem{bershtein}{\sc O.Bershtein}, {\it On $*$-representations of polynomial algebras in quantum matrix space of rank 2},
Algebras and Representation Theory. Online publication 2013.

\bibitem{jsw}{\sc  P.~E.~T.~J{\o }rgensen, L.~M.~Schmitt and R.~F.~Werner},
{\it Positive representations of general commutation relations allowing
Wick ordering}, {\rm J. Funct. Anal. 134 (1995),  33--99.}


\bibitem{hamana} {\sc M.Hamana}, {\it Injective envelopes of operator systems}, {\rm Publ.RIMS Kyoto Univ. (1979), v.15, 773--785}
\bibitem{soibelman-vaksman}{\sc L. Vaksman and Y. Soibelman}, {\it Algebra of functions on the quantum group SU(2)},
{\rm Funct. Anal. Appl. 22 (1988) 170--181}.

\bibitem{OS}{\sc V.Ostrovskyi and Yu.Samoilenko}, Introduction to the Theory Of Representations
of Finitely Presented Algebras, Rev. Math. $\&$ Math. Phys. {\bf 11}, London: Gordon $\&$ Breach, (2000) 261 p.

\bibitem{pro_tur} {\sc D.Proskurin and L.Turowska}, {\it On the $C^*$-algebra associated with $Pol(Mat_{2,2})_q$}, {\rm Methods Funct. Anal. Topology  7  (2001),  no. 1, 88--92.}

\bibitem{pusz-woronowicz} {\sc W.Pusz and S.L.Woronowicz} {\it Twisted second quantization}, {\rm  Rep. Math. Phys. 27 (1989), no. 2, 231–-257.}

\bibitem{ssv}{\sc D.Shklyarov, S.Sinelʹshchikov and L.Vaksman}, {\it Fock representations and quantum matrices}, {\rm Internat. J. Math.  15  (2004),  no. 9, 855--894}.

%\bibitem{ssv-dif}{\sc  S.Sinel'shchikov and L.Vaksman}, {\it Hidden symmetry of the differential calculus on the quantum matrix space}, {\sc J. Phys. A 30 (1997), no. 2, L23–L26.}, {\rm}.

\bibitem{sv}{\sc  S.Sinelshchikov and L.Vaksman} {\it On q-analogues of bounded symmetric domains and Dolbeault complexes},
{\rm Math. Phys. Anal. Geom.  1  (1998),  no. 1, 75–-100.}

\bibitem{ssv-dif}{\sc D.Shklyarov, S.Sinelshchikov and  L.Vaksman} {\it Quantum matrix ball: differential and integral calculi},
{\rm  arXiv:math/9905035. }

\bibitem{T} {\sc L.Turowska}, {\it Representations of a q-analogue of the $*$-algebra $Pol(Mat_{2,2})$},
{\rm J. Phys. A  34  (2001),  no. 10, 2063–-2070.
}

\bibitem{vaksman-matrix ball} {\sc L.Vaksman},{\it Quantum matrix ball: the Cauchi-Szeg\"o kernel and the Shilov boundary}, {\rm Mat. Fiz. Anal. Geom. 8 (2001),  no. 4,  366–-384.}

\bibitem{vaksman-boundary}{\sc L. Vaksman}, {\it Maximum principe for "holomorphic functions" in the quantum ball}, {\rm Mat. Fiz. Anal. Geom. 10 (2003), no. 1,  12–-28.}
\bibitem{vaksman-book} {\sc L. Vaksman}, {\it Quantum Bounded Symmetric Domains}, {\rm Amer.Math.Soc., Providence RI, 2010}.
\end{thebibliography}
\end{document}